\documentclass[11pt]{amsart}

\usepackage[T1]{fontenc}
\usepackage[utf8]{inputenc}
\usepackage{lmodern}
\usepackage{microtype}

\usepackage{amsmath,amssymb,mathtools}

\usepackage[margin=1in]{geometry}

\usepackage{aliascnt}
\usepackage[
    colorlinks=true,
    linkcolor=blue,
    citecolor=blue,
    urlcolor=blue
]{hyperref}
\usepackage[nameinlink,capitalize,noabbrev]{cleveref}

\theoremstyle{plain}

\newtheorem{theorem}{Theorem}[section]

\newaliascnt{proposition}{theorem}
\newtheorem{proposition}[proposition]{Proposition}
\aliascntresetthe{proposition}

\newaliascnt{lemma}{theorem}
\newtheorem{lemma}[lemma]{Lemma}
\aliascntresetthe{lemma}

\newaliascnt{corollary}{theorem}
\newtheorem{corollary}[corollary]{Corollary}
\aliascntresetthe{corollary}

\theoremstyle{definition}

\newaliascnt{definition}{theorem}
\newtheorem{definition}[definition]{Definition}
\aliascntresetthe{definition}

\theoremstyle{remark}

\newaliascnt{remark}{theorem}
\newtheorem{remark}[remark]{Remark}
\aliascntresetthe{remark}

\crefname{theorem}{theorem}{theorems}
\crefname{proposition}{proposition}{propositions}
\crefname{lemma}{lemma}{lemmas}
\crefname{corollary}{corollary}{corollaries}
\crefname{definition}{definition}{definitions}
\crefname{remark}{remark}{remarks}

\newcommand{\bbE}{\mathbb E}
\newcommand{\bbP}{\mathbb P}
\newcommand{\calB}{\mathcal B}
\newcommand{\calC}{\mathcal C}
\newcommand{\calH}{\mathcal H}
\newcommand{\calN}{\mathcal N}
\newcommand{\calR}{\mathcal R}
\newcommand{\calX}{\mathcal X}
\newcommand{\calY}{\mathcal Y}
\newcommand{\cost}{\operatorname{cost}}
\newcommand{\covol}{\operatorname{covol}}
\newcommand{\inte}{\operatorname{int}}
\newcommand{\eint}{\mathfrak e}
\newcommand{\one}{\mathbf 1}

\title{A Product-Neighbourhood Criterion for Fixed Price One}
\author{Raz Slutsky}
\address{Merton College, University of Oxford, Merton Street, Oxford, OX1 4JD, United Kingdom.}
\email{raz.slutsky@maths.ox.ac.uk}
\date{}

\begin{document}

\begin{abstract}
We prove a flexible criterion for fixed price one, applying in particular to
higher-rank lattices over local fields and automorphism groups of
affine buildings, including lattices in exotic buildings. It also recovers fixed price one for amenable groups. We then establish a locally compact version, and
deduce that every product of two noncompact, compactly generated, unimodular
locally compact groups, as well as every lattice in such a product, has fixed price one.
\end{abstract}

\maketitle

\section{Introduction}

Cost is a numerical invariant of probability-measure-preserving (pmp) actions, introduced by Levitt \cite{Levitt} and greatly developed by
Gaboriau \cite{GaboriauCost,GaboriauL2}. For a countable group, the cost of an essentially free pmp action measures
the least average number of generators needed to generate the associated
orbit equivalence relation. See \Cref{def:cost} for the formal definition. Cost has found important applications to orbit-equivalence classification, to the study of \(\ell^2\)-Betti numbers, to asymptotic questions in group theory, to operator algebras, and more. See, for example, \cite{AbertNikolov, GaboriauCost,GaboriauL2, GaboriauLyons, ShlyakhtenkoFisher}. A group
has \emph{fixed price} if all of its essentially free pmp actions have the same
cost; it has \emph{fixed price one} if this common value is one.  The fixed price
problem, asking whether every countable group has fixed price, remains one of
the basic open questions in measured group theory.

There is now a parallel theory for locally compact second countable unimodular
groups.  This appears in work of Kyed--Petersen--Vaes
\cite{KyedPetersenVaes} and Carderi \cite{Carderi}, and was  
developed systematically by Ab\'ert and Mellick in
\cite{AbertMellick, MellickPalmGroupoid}.  In particular, they connected cost to the theory of point
processes, and showed that Poisson processes have maximal cost.  This approach has since
been used to prove fixed price one in several higher-rank situations: by
Fr\k{a}czyk--Mellick--Wilkens for lattices in higher rank semisimple real Lie
groups 
\cite{FraczykMellickWilkens} and by Mellick \cite{MellickGaboriauCriterion} in certain products of semisimple algebraic groups.  In the
discrete case, Khezeli
\cite{Khezeli} recently showed that a direct product of discrete groups has fixed price one. See also related work of Bevilacqua and Bowen \cite{BB}.

The purpose of this paper is to establish a flexible criterion for fixed
price one that applies in both the discrete and locally compact settings.
The criterion provides a uniform proof of the cases discussed above and
applies to a broad range of additional groups.

We first prove the following discrete version.

\begin{theorem}\label{thm:discrete}
Let $\Gamma$ be a discrete, finitely generated group, and let $S$ be a finite
symmetric generating set with $e\in S$.  Suppose that there are finite sets
$F_n\subseteq\Gamma$, with $e\in F_n$, such that
\[
        \frac{|F_nSF_n^{-1}|}{|F_n|^2}\longrightarrow 0.
\]
Then $\Gamma$ has fixed price one.
\end{theorem}

As an application, the criterion applies to a broad class of higher-rank lattices and to direct products of discrete groups, yielding fixed price one in both settings. It is also satisfied by amenable groups, Baumslag-Solitar groups, $\operatorname{MCG}(S_g)$, and more.

This also gives a direct proof of Gaboriau's weak normality criterion
\cite[Crit\`eres~VI.24(3)]{GaboriauCost}.

\begin{corollary}\label{cor:weak-normality}
Let \(\Gamma\) be a finitely generated group. Suppose that \(\Gamma\)
contains an infinite amenable subgroup \(\Lambda\) such that
\(\Lambda\cap s^{-1}\Lambda s\) is infinite for every \(s\in S\).
Then \(\Gamma\) has fixed price one.
\end{corollary}
Here we require a weaker form of normality, only with respect to $S$, known as $q$-normality
\cite[Section~5.1]{PetersonThom}.
See also \Cref{cor:finite-subgroups} for a version for finite subgroups.
This proves, for example, that Thompson's group $V$ and other topological full groups have fixed price one. 

The locally compact version is as follows.  

\begin{theorem}\label{thm:lc-criterion}
Let $G$ be a noncompact, compactly generated, unimodular locally compact
second countable group, let $\lambda$ be a left Haar measure, and let
$e \in S = S^{-1}\subseteq G$ be a compact  generating set.
Suppose that there are compact Borel sets $F_n\ni e$ 
such that
\[
        \frac{\lambda(F_nSF_n^{-1})}{\lambda(F_n)^2}\longrightarrow 0.
\]
Then $G$ has fixed price one and so does every lattice in $G$.
\end{theorem}

The criterion has the following metric consequence.

\begin{corollary}\label{cor:lc-metric}
Let $G$ admit a proper left-invariant metric with
\(
        D=\max_{s\in S}d(e,s)
\) and such that
\[
        \liminf_{R\to\infty}\frac{\lambda(B(2R+D))}{\lambda(B(R))^2}=0,
\]
then \(G\) has fixed price one, and so does every lattice in \(G\).
\end{corollary}

For instance, the hypothesis holds whenever
\(
        \lambda(B(R))\asymp e^{hR}R^a
        \)
       with $a>0$, and by the \v{S}varc--Milnor lemma, this also applies to any group acting co-compactly and properly discontinuously on spaces with such volume growth. 

\begin{corollary}\label{cor:local-field-semisimple} Let \(k\) be a local field and let \(\mathbf G\) be a connected semisimple \(k\)-group with \( \operatorname{rank}_{k}\mathbf G \geq 2\) Then \(G = \mathbf{G}(k)\) has fixed price one, and so does every lattice in \(G\). \end{corollary}

Thus the result covers both Archimedean and non-Archimedean higher-rank
groups, including those over local fields of positive characteristic. We also get the following for higher-dimensional affine buildings, including exotic buildings.

\begin{corollary}\label{cor:affine-buildings}
Let \(X\) be a locally finite thick regular affine building of Euclidean rank
\(r\ge2\), and let \(G<\operatorname{Aut}(X)\) be a closed unimodular subgroup
acting cocompactly on \(X\). Then \(G\) has fixed price one, and so does every
lattice in \(G\). 
\end{corollary}
In particular, this gives fixed price one for the recently constructed
simple, non-residually finite uniform lattices on exotic buildings of
type \(\widetilde C_2\)
\cite{TitzMiteWitzel,LecureuxWitzel}. We note that when the affine Coxeter system is \(2\)-spherical, regularity is
automatic \cite[Section~1.5]{AbramenkoParkinsonVanMaldeghem}. In particular, this applies to every irreducible affine type of Euclidean rank at least two. As an additional application, we get:

\begin{corollary}\label{cor:lc-products}
Let $G_1$ and $G_2$ be noncompact, compactly generated, unimodular lcsc groups.
Then $G_1\times G_2$ has fixed price one and so does every lattice in such a product.
\end{corollary}
This
answers a question of Ab\'ert--Mellick
\cite[Question~2]{AbertMellick}, where the case
\(G\times\mathbb Z\) was established, and extends
\cite{MellickGaboriauCriterion}, which treats certain products of semisimple
algebraic groups.

We record, by \cite[Corollaire~3.23]{GaboriauL2} and \cite{AbertNikolov, AbertToth}, the following:

\begin{corollary}\label{cor:rank-gradient-l2}
Let $\Gamma$ be a finitely generated discrete group which is shown by one of
the preceding criteria to have fixed price one.  Then
\(
        \beta_1^{(2)}(\Gamma)=0.
\)
Moreover, for every Farber sequence $(\Gamma_n)$ of finite-index subgroups of
$\Gamma$,
\[
        \operatorname{RG}(\Gamma,(\Gamma_n))
        :=
        \lim_{n\to\infty}
        \frac{d(\Gamma_n)-1}{[\Gamma:\Gamma_n]}
        =0,
\]
where $d(\Lambda)$ denotes the minimal number of generators of $\Lambda$.
\end{corollary}

The criterion also gives a direct proof of the classical fixed-price-one
theorem for amenable groups:

\begin{corollary}\label{cor:amenable-groups}
Every infinite finitely generated amenable group satisfies the hypothesis of
\Cref{thm:discrete}.  In particular, it has fixed price one.
\end{corollary}
The argument for the above uses only the F{\o}lner criterion for amenability and does not invoke the
Ornstein--Weiss theorem.

The common idea behind the proofs is to construct a sparse connected equivariant graph using the sets \(F_n\). In the discrete case, we start from the Cayley graph and use the Abért--Weiss theorem that Bernoulli actions have maximal cost \cite{AbertWeiss}. The locally compact argument uses the point-process formulation of cost developed by Abért and Mellick, together with several tools from point-process theory. See also related sparse factor constructions in \cite{GrebikRecke}. For background, we refer the reader to \cite{AbertMellick} and to the book \cite{LastPenrose}. 

\paragraph{\textbf{Acknowledgments}}
We are very grateful to Konstantin Recke for useful discussions and suggestions regarding the locally-compact case and to Miko\l{}aj Fr\k{a}czyk and Sam Mellick for helpful comments.

\section{The discrete case}

Throughout this section, $\Gamma$ is an infinite finitely generated group and
$S$ is a finite symmetric generating set with $e\in S$.

\begin{definition}[Cost]\label{def:cost}
Let \((X,\mu)\) be a standard probability space and let \(\mathcal R\) be a
countable pmp equivalence relation on \(X\). A \emph{graphing} of
\(\mathcal R\) is a countable family
\[
    \Phi=\{\varphi_i:A_i\longrightarrow B_i\}_{i\in I}
\]
of measure-preserving Borel isomorphisms between Borel subsets of \(X\) such
that
\[
 (x,\varphi_i(x))\in\mathcal R
    \qquad
    \text{for every \(i\in I\) and \(x\in A_i\)}.
\]
The graphing \(\Phi\) defines a graph on \(X\) by joining
\(x\) to \(\varphi_i(x)\). We say that
\(\Phi\) \emph{generates} \(\mathcal R\) if, on a conull
subset of \(X\), the connected components of
this graph are precisely the \(\mathcal R\)-classes.

The cost of \(\Phi\) is
\[
    \cost(\Phi)=\sum_{i\in I}\mu(A_i),
\]
and the cost of \(\mathcal R\) is
\[
    \cost(\mathcal R)
    =
    \inf\bigl\{
        \cost(\Phi):
        \Phi\text{ is a graphing that generates }\mathcal R
    \bigr\}.
\]
For an essentially free pmp action
\(\Gamma\curvearrowright(X,\mu)\), we write
\[
    \cost(\Gamma\curvearrowright X)
    :=
    \cost(\mathcal R_{\Gamma\curvearrowright X}),
\]
where \(\mathcal R_{\Gamma\curvearrowright X}\) denotes its orbit equivalence
relation.
\end{definition}

\subsection{Proof of the criterion}

\begin{proof}[Proof of \Cref{thm:discrete}]
Fix a finite set $F\subseteq\Gamma$ with $e\in F$, and a parameter
$p\in(0,1)$.  Let
\(
        \Omega=\{0,1\}^{\Gamma}
\)
with Bernoulli product measure of parameter $p$, equipped with the Bernoulli
shift.  For $\omega\in\Omega$, consider the set
\[
        Y(\omega)=\{g\in\Gamma:\omega(g)=1\}.
\]
We think of the elements of $Y(\omega)$ as those elements which were selected
by a random process.

A point $g\in\Gamma$ is called \emph{bad} if
\[
        Y(\omega)\cap gF^{-1}=\varnothing.
\]
Let $C(\omega)$ be the union of the selected points and the bad points.  These
two subsets are disjoint because $e\in F$.  For $c\in C(\omega)$ define the
patch
\[
        U_c(\omega)=
        \begin{cases}
        cF, & \text{if $c$ is selected},\\
        \{c\}, & \text{if $c$ is bad}.
        \end{cases}
\]
These patches cover $\Gamma$: if $g$ is bad then $g\in U_g(\omega)$, while if
$g$ is not bad, then some selected $y\in Y(\omega)\cap gF^{-1}$ satisfies
$g\in yF$.  The construction is equivariant:
\[
        C(\gamma\omega)=\gamma C(\omega),
        \qquad
        U_{\gamma c}(\gamma\omega)=\gamma U_c(\omega).
\]

Define a graph $\calN(\omega)$ with vertex set $C(\omega)$. Two distinct
vertices $c,c'$ are adjacent if the corresponding patches are at distance at
most one. In other words, if \(U_c(\omega)s\cap
U_{c'}(\omega)\neq\varnothing\) for some \(s\in S\).
Since the patches cover the Cayley graph, the graph
$\calN(\omega)$ is connected.

Let $\calR_p$ be the orbit equivalence relation of the Bernoulli shift, and set
\[
        A=\{\omega:e\in C(\omega)\}
\]
Then
\(
        \bbP(A)=p+(1-p)^{|F|}>0.
\)
Moreover, $A$ is a complete section in the sense that it meets every equivalence class. Indeed, almost surely there is a selected point
$y$, and then $e$ is selected in $y^{-1}\omega$.

For $\gamma\neq e$, let $A_\gamma$ be the set of all $\omega\in A$ such that
$\gamma\in C(\omega)$ and $e,\gamma$ are adjacent in $\calN(\omega)$.  On
$A_\gamma$, define the partial isomorphism
\[
        \varphi_\gamma(\omega)=\gamma^{-1}\omega.
\]
Equivariance implies that $\varphi_\gamma(\omega)\in A$.  Put
\[
        \eta(p,F)=\sum_{\gamma\neq e}\bbP(A_\gamma).
\]

The partial isomorphisms $\{\varphi_\gamma\}_{\gamma\neq e}$ generate
$\calR_p|_A$.  Indeed, suppose $\theta=\alpha\omega$ with
$\omega,\theta\in A$.  Since
$e\in C(\alpha\omega)=\alpha C(\omega)$, we have
$\alpha^{-1}\in C(\omega)$.  By connectedness of $\mathcal N(\omega)$, choose a path
\[
        e=c_0,c_1,\ldots,c_m=\alpha^{-1}
\]
in $\calN(\omega)$.  At the $i$th step put
\[
        \gamma_i=c_i^{-1}c_{i+1}.
\]
In $c_i^{-1}\omega$, the vertices $e$ and $\gamma_i$ are adjacent centres, so
$\varphi_{\gamma_i}$ is defined there and sends $c_i^{-1}\omega$ to
$c_{i+1}^{-1}\omega$.  Iterating gives
\[
        \omega\longmapsto c_m^{-1}\omega=\alpha\omega=\theta.
\]

Hence $\calR_p|_A$ has a graphing of cost
$\eta(p,F)/\bbP(A)$ with respect to the normalized measure on $A$.  By
Gaboriau's compression formula \cite[Proposition~II.6]{GaboriauCost},
\[
\begin{aligned}
\cost(\calR_p)-1
&=\bbP(A)\bigl(\cost(\calR_p|_A)-1\bigr)
  \leq \bbP(A)\cost(\calR_p|_A)
  \leq \eta(p,F).
\end{aligned}
\]
Thus
\[
        \cost(\calR_p)\leq 1+\eta(p,F).
\]

It remains to estimate $\eta(p,F)$.  We split possible edges in $\mathcal{N}(\omega)$ according
to whether their two endpoints are selected or bad.

If both $e$ and $\gamma$ are selected, then adjacency of $F$ and $\gamma F$
implies
\(
        \gamma\in FSF^{-1}.
\)
The probability that both $e$ and $\gamma$ are selected is $p^2$, so the total
selected--selected contribution is at most
\[
        p^2|FSF^{-1}|.
\]

If both endpoints are bad, then their patches are singletons, so adjacency
forces $\gamma\in S$.  Since
\[
        \bbP(e\text{ is bad})=(1-p)^{|F|}\leq e^{-p|F|},
\]
the bad--bad contribution is at most
\[
        |S|e^{-p|F|}
\]

If $e$ is bad and $\gamma$ is selected, then adjacency gives
$\gamma f=a$ for some $f\in F$ and $a\in S$, hence
\(
        \gamma\in SF^{-1}.
\)
Apart from the incompatible cases where $\gamma\in F^{-1}$, badness of $e$
and selection of $\gamma$ are independent.  Hence this contribution is at
most
\[
        pe^{-p|F|}|SF^{-1}|.
\]
Similarly, the selected--bad contribution is at most
\[
        pe^{-p|F|}|FS|.
\]
Therefore
\[
\eta(p,F)
\leq
p^2|FSF^{-1}|
+|S|e^{-p|F|}
+pe^{-p|F|}\bigl(|SF^{-1}|+|FS|\bigr).
\]

Now assume that $F=F_n$ satisfies the hypothesis of the theorem, and set
\[
        \alpha_n=\frac{|F_nSF_n^{-1}|}{|F_n|^2}.
\]
Since $e\in F_n\cap S$, the set $F_nSF_n^{-1}$ contains $F_n$, so
\(
        \alpha_n\geq |F_n|^{-1}.
\)
Define
\[
        a_n=\alpha_n^{-1/4},
        \qquad
        p_n=\frac{a_n}{|F_n|}.
\]
Then $p_n<1$ for large $n$, while
\[
        p_n|F_n|=a_n\to\infty,
        \qquad
        a_n^2\alpha_n=\alpha_n^{1/2}\to 0.
\]
Using
\[
        |SF_n^{-1}|,\ |F_nS|\leq |S||F_n|,
\]
the preceding estimate gives
\[
\eta(p_n,F_n)
\leq
 a_n^2\alpha_n
 +|S|e^{-a_n}
 +2|S|a_ne^{-a_n}
\longrightarrow 0.
\]
Therefore the costs of the Bernoulli shifts with parameters $p_n$ tend to one
from above.

Finally, by the Ab\'ert--Weiss theorem, \cite[Corollary~2]{AbertWeiss}, Bernoulli
actions have maximal cost.  Hence for every essentially free pmp action
$\alpha$ of $\Gamma$,
\[
        \cost(\alpha)
        \leq \cost(\calR_{p_n})
        \leq 1+\eta(p_n,F_n).
\]
Letting $n\to\infty$ gives $\cost(\alpha)\leq1$.  Since $\Gamma$ is infinite, $\cost(\alpha)=1$, and $\Gamma$ has fixed price one.
\end{proof}

\begin{remark}
There is also a direct version of the preceding argument in which the
vertex set remains all of $\Gamma$.  Fix an arbitrary ordering on
\(
        F^{-1}\).
For each non-bad $g$, let
$a_\omega(g)$ be the first selected point in $gF^{-1}$. 
The map $a_\omega$ is equivariant.
Define a graph $\mathcal G(\omega)$ on $\Gamma$ as follows.  Join every
non-bad $g$ to $a_\omega(g)$; join two distinct selected points 
whenever their $F$-patches are at Cayley distance at most one,
and retain every Cayley edge having at least one bad endpoint. 
It is not hard to see that this graph is connected. 
Since $\mathcal G$ is an equivariant connected graph on the whole orbit,
\[
        \cost(\calR_p)
        \leq
        \frac12\bbE\bigl[\deg_{\mathcal G(\omega)}(e)\bigr],
\]
and the same probabilistic estimates bound the expected degree of the identity.
\end{remark}

\subsection{Applications}

We start with direct products.

\begin{lemma}\label{lem:product-neighbourhoods}
For \(i=1,2\), let \(H_i\) be a noncompact compactly generated lcsc group,
let \(\mu_i\) be a left Haar measure, and let \(K_i\) be a compact symmetric
generating neighbourhood.  Put
\(
        \mu=\mu_1\times\mu_2,
\) and \(
        S=K_1\times K_2.
\)
Then there are compact sets \(F_N\ni(e,e)\) such that
\[
        \frac{\mu(F_NSF_N^{-1})}{\mu(F_N)^2}\longrightarrow0.
\]
\end{lemma}

\begin{proof}
For each \(i\), put \(u_i=\mu_i(K_i)\), and define
\[
        b_i(n)=u_i^{-1}\mu_i(K_i^{n+2}).
\]
Then \(b_i\) is nondecreasing, unbounded, and submultiplicative.  Indeed, if
\(z=ab\in K_i^{m+n+2}\), where \(a\in K_i^{m+1}\) and
\(b\in K_i^{n+1}\), then \(aK_i\) contributes \(u_i\) to
\(
        (\one_{K_i^{m+2}}*\one_{K_i^{n+2}})(z),
\)
and integration gives \(b_i(m+n)\leq b_i(m)b_i(n)\). For \(x\in H_i\), set
\[
        \ell_i(x)=\min\{n\geq0:x\in K_i^{n+2}\},
        \qquad
        w_i(x)=b_i(\ell_i(x)).
\]
Writing
\(
        V_i(t)=\mu_i\{x:w_i(x)\leq t\},
\)
we have, for \(t\geq b_i(0)\),
\begin{equation}\label{eq:gauge-volume}
        V_i(t)\asymp t.
\end{equation}
Indeed, if \(r=\max\{n:b_i(n)\leq t\}\), then
\(V_i(t)=u_ib_i(r)\), while maximality and
\(b_i(r+1)\leq b_i(r)b_i(1)\) give
\(
        \frac{t}{b_i(1)}<b_i(r)\leq t.
\)
Moreover, if \(x,y\in H_i\) and \(k\in K_i\), then
\begin{equation}\label{eq:gauge-product}
        w_i(xky^{-1})\leq b_i(3)w_i(x)w_i(y).
\end{equation}
Indeed, if \(\ell_i(x)=m\) and \(\ell_i(y)=n\), then
\(xky^{-1}\in K_i^{m+n+5}\), and submultiplicativity applies.

Put
\[
        W(x_1,x_2)=w_1(x_1)w_2(x_2),
        \qquad
        F_T=\{(x_1,x_2):W(x_1,x_2)\leq T\}.
\]
These sets are finite unions of compact rectangles, hence compact.
By \eqref{eq:gauge-volume}, we get
\[
        \int_{\{w_i\leq R\}}\frac{d\mu_i}{w_i}\asymp\log R.
\]
Consequently, by Fubini,
\[
\begin{aligned}
        \mu(F_T)
        =\int_{H_1}
          V_2\!\left(\frac{T}{w_1(x)}\right)d\mu_1(x) 
        \asymp
          T\int_{\{w_1\leq T/b_2(0)\}} \frac{d\mu_1(x)}{w_1(x)}
         \asymp T\log T.
\end{aligned}
\]
On the other hand, \eqref{eq:gauge-product} gives, with
\(C=b_1(3)b_2(3)\),
\(F_TSF_T^{-1}\subseteq F_{CT^2}.
\)
Therefore
\[
        \frac{\mu(F_TSF_T^{-1})}{\mu(F_T)^2}
        \ll
        \frac{T^2\log T}{T^2(\log T)^2}
        \longrightarrow0.
\]
Taking \(F_N=F_{T_N}\), where \(T_N\to\infty\) and
\(T_N\geq W(e,e)\), proves the lemma.
\end{proof}

We prove the finitely generated case. The general case follows by taking an increasing union of finitely generated subgroups.

\begin{corollary}\label{cor:discrete-products}
Let $G$ and $H$ be infinite finitely generated groups.  Then $G\times H$ has
fixed price one.
\end{corollary}

\begin{proof}
Choose finite symmetric generating sets \(K_G,K_H\), and put \(S=K_G\times K_H\).  By
\Cref{lem:product-neighbourhoods}, there are finite sets
\(F_N\subseteq G\times H\), with \((e,e)\in F_N\), such that
\Cref{thm:discrete} applies to \(G\times H\).
\end{proof}


\begin{corollary}\label{cor:finite-subgroups}
Suppose that $K_n<\Gamma$ are finite subgroups such that
\[
        \min_{s\in S}|K_n\cap sK_ns^{-1}|\longrightarrow\infty.
\]
Then $\Gamma$ has fixed price one.
\end{corollary}

\begin{proof}
This follows since for subgroups, we have \(|K_nsK_n|=|K_n|^2/|K_n\cap sK_ns^{-1}|\) for every \(s\in S\).
\end{proof}

\subsection{Amenable groups}
\begin{proof}[Proof of \Cref{cor:amenable-groups}]
We use the following Ruzsa triangle inequality: for finite 
\(A,B,C\subseteq\Gamma\),
\[
        |AC^{-1}|\leq \frac{|AB^{-1}|\,|BC^{-1}|}{|B|}.
\]
Fix \(M>0\).  Choose a finite set \(K\subseteq\Gamma\) with \(|K|\geq M\).
By amenability, choose a finite nonempty set \(E\subseteq\Gamma\) such that
\[
        |E(K^{-1}\cup SK^{-1})|\leq 2|E|.
\]
Then \(|EK^{-1}|\leq2|E|\) and \(|ESK^{-1}|\leq2|E|\).  Applying Ruzsa with
\(A=ES\), \(B=K\), and \(C=E\), we get
\[
        |ESE^{-1}|
        \leq
        \frac{|ESK^{-1}|\,|KE^{-1}|}{|K|}
        =
        \frac{|ESK^{-1}|\,|EK^{-1}|}{|K|}
        \leq
        \frac{4}{M}|E|^2.
\]
Choose \(h\in E\) and put \(F=h^{-1}E\).  Then \(e\in F\) and
\(
        |FSF^{-1}|=|ESE^{-1}|.
\)
Since \(M\) was arbitrary, the sparse product-neighbourhood hypothesis follows.
\end{proof}

\begin{proof}[Proof of \Cref{cor:weak-normality}]
Fix \(M\geq1\). For each \(s\in S\), choose a finite set
\(K_s\subseteq\Lambda\cap s^{-1}\Lambda s\) with \(|K_s|\geq M\).
By amenability of \(\Lambda\), there is a finite nonempty
\(E\subseteq\Lambda\) such that
\[
 |EK_s^{-1}|,\ \bigl|E(sK_s^{-1}s^{-1})\bigr|\leq 2|E|
 \qquad (s\in S).
\]
The Ruzsa triangle inequality from the preceding proof, applied to
\(Es,K_s,E\), gives
\[
 |EsE^{-1}|
 \leq
 \frac{\bigl|E(sK_s^{-1}s^{-1})\bigr|\,|EK_s^{-1}|}{|K_s|}
 \leq \frac{4}{M}|E|^2 .
\]
Choose \(h\in E\) and put \(F=h^{-1}E\), so that \(e\in F\). Summing over
\(s\in S\), we obtain
\[
 \frac{|FSF^{-1}|}{|F|^2}
 =\frac{|ESE^{-1}|}{|E|^2}
 \leq \frac{4|S|}{M}.
\]
Since \(M\) is arbitrary, \Cref{thm:discrete} applies.
\end{proof}

\section{The locally compact case}

Throughout this section, $G$ is a noncompact, nondiscrete, compactly
generated, unimodular lcsc group, $\lambda$ is a
left Haar measure, and $S\subseteq G$ is a compact symmetric generating
neighbourhood of the identity.  For a compact Borel set $F\ni e$, set
\(
        \alpha_S(F)=\frac{\lambda(FSF^{-1})}{\lambda(F)^2}.
\)
The cost of a Poisson process is independent of its positive intensity by
\cite[Theorem~1.2]{AbertMellick}; denote this value by $c_{\mathrm P}(G)$.

\subsection{Cost and a separated process}

Fix a proper compatible left-invariant metric on $G$. For an invariant marked point process $\Pi$ of finite positive intensity and a
factor graph $\mathcal G$ on $\Pi$, let
\[
\inte(\Pi)=\bbE|\Pi\cap U|,
\qquad
\eint(\mathcal G)=\frac12\bbE\sum_{x\in\Pi\cap U}
                       \deg_{\mathcal G}(x),
\]
where $U$ is any Borel set of Haar measure one.  These quantities are
independent of $U$.  In the point-process formulation of cost,
\begin{equation}\label{eq:cost}
 \cost(\Pi)-1=
 \inf_{\mathcal G}\bigl(\eint(\mathcal G)-\inte(\Pi)\bigr),
\end{equation}
where the infimum runs over factor graphs whose vertex set is all of
\(\Pi\) 
\cite[Definition~4.1]{AbertMellick}.

Suppose a standard pmp action $G\curvearrowright(\Omega,\mu)$ produces an
equivariant, almost surely nonempty, locally finite set
$C(\omega)\subseteq G$ of finite intensity.  Fix a Borel injection
$I:\Omega\to[0,1]$ and mark $c\in C(\omega)$ by $I(c^{-1}\omega)$.  Any one
marked point recovers $\omega$, so the resulting marked process is isomorphic
to the original action.  This is the label-trickery construction of
\cite[Proposition~4.13]{AbertMellick}. The following proposition is contained in \cite{AbertMellick} and provides the analogue of the Cayley graph in this setting.

\begin{proposition}\label{prop:scaffold}
A unit-intensity Poisson process on $G$ has a factor consisting of a locally
finite point process $\calX$ and a locally finite graph $\calH$ on $\calX$
such that
\[
0<\inte(\calX)<\infty,
\qquad
h:=\eint(\calH)<\infty,
\]
$\calH$ is connected, and $x^{-1}x'\in S$ for every oriented edge
$(x,x')$ of $\calH$.
\end{proposition}

\begin{proof}
A Poisson process is free and ergodic
\cite[Proposition~2.7]{AbertMellick}.  Corollary~4.12 and Propositions~4.13 and
4.18 of \cite{AbertMellick} give, up to marking, an equivariant Delone set
$\Delta$; that is, almost surely, $\Delta$ is uniformly separated and coarsely
dense.  Hence the graph obtained by connecting points in $\Delta$ if they are
at distance at most $r$ is connected for some $r$.  Call it $\calH_0$.
Uniform separation makes its degree bounded, and the elements corresponding
to its edges lie in a compact set $Q$.

Choose $k$ with $Q\subseteq S^k$.  The multiplication map
$m:S^{\times k}\to S^k$ has nonempty compact fibres over $Q$, so the
Arsenin--Kunugui theorem \cite[Theorem~18.18]{Kechris} gives Borel maps
$s_i:Q\to S$ satisfying
$q=s_1(q)\cdots s_k(q)$.  Replace each ordered edge $(x,xq)$ of $\calH_0$ by
the corresponding $S$-path.  Write $K=S^k$.  Every path vertex lying in a
relatively compact set $U$ comes from a base point of
$\Delta\cap UK^{-1}$, and each base point starts only boundedly many path
edges.  Since $\Delta$ has finite intensity, the resulting process $\calX$ is
locally finite of finite positive intensity, and the connected graph $\calH$
is locally finite with finite edge intensity.
\end{proof}

\subsection{Proof of the locally compact criterion}

\begin{proof}[Proof of \Cref{thm:lc-criterion}]
Let $F_n$ be compact Borel sets containing $e$ with
$\alpha_S(F_n)\to0$.  We first prove that there is a constant
$A=A(G,S)<\infty$ such that, for every compact Borel set $F\ni e$ of positive
measure and every $t>0$,
\begin{equation}\label{eq:estimate}
 c_{\mathrm P}(G)-1\leq \frac{t^2}{2}\alpha_S(F)+Ae^{-t}.
\end{equation}
Fix $F$ and $t$, put $v=\lambda(F)$ and $p=t/v$, and let $Z$ and $\calY$ be
independent Poisson processes of intensities $1$ and $p$.  They can be
realized as a Poisson process $\Omega$ of intensity $1+p$, with the points of
$Z$ colored red and the points of $\calY$ colored blue, see
\cite[Theorems~5.6 and~5.8]{LastPenrose}.  IID marking preserves cost
\cite[Theorem~2.17]{FraczykMellickWilkens}, while Poisson cost is independent
of intensity; hence
\(
        \cost(\Omega)=c_{\mathrm P}(G).
\)
Now, construct $(\calX,\calH)$ from $Z$ using \Cref{prop:scaffold}.

We fix a Borel injection $\tau:G\to(0,1)$.  Call $x\in\calX$ \emph{bad} if
$\calY\cap xF^{-1}=\varnothing$, and write $\calB$ for the bad vertices.
Define a map $a:\calX\to G$ by $a(x)=x$ for $x\in\calB$, while for
$x\notin\calB$ let $a(x)$ be the unique $y\in\calY\cap xF^{-1}$ minimizing
$\tau(x^{-1}y)$.  The candidate set is finite, so this is well defined, and
the rule is measurable and equivariant.  We call $a(x)$ the \emph{anchor} of
$x$.

Set $\calC=a(\calX)$.  If $y\in\calY$ is an anchor, then
$a^{-1}(y)\subseteq\calX\cap yF$, so its fibre is finite.  If $b\in\calB$,
then $e\in F$ implies $b\notin\calY$, and hence $a^{-1}(b)=\{b\}$.  Thus the
anchor map is finite-to-one on the incidence space; its image is Borel by the
Lusin--Novikov theorem \cite[Theorem~18.10]{Kechris}.  Moreover,
$\calC\subseteq\calY\cup\calB$.  By unimodularity,
$\lambda(F^{-1})=v$, and independence yields
\[
        \inte(\calB)=\inte(\calX)e^{-pv}
                    =\inte(\calX)e^{-t}.
\]
Hence $\calC$ is an almost surely nonempty locally finite process of finite
positive intensity.  Mark it as described above, obtaining a process
$\widehat{\calC}$ isomorphic to $\Omega$.

Define a graph $\calN$ on the vertices of $\calC$.  Join distinct
$c,c'\in\calC$ whenever some edge $\{x,x'\}$ of $\calH$ satisfies
$a(x)=c$ and $a(x')=c'$.  Since the fibres are finite, $\calN$ is locally
finite, and projecting an $\calH$-path through $a$ shows that $\calN$ is
connected.  The marks recover $\Omega$, so $\calN$ is a factor graph of
$\widehat{\calC}$.

Suppose two blue anchors $y,y'$ are adjacent, witnessed by an edge
$(x,x')$ of $\calH$.  Then
\[
 y^{-1}y'=(y^{-1}x)(x^{-1}x')(x'^{-1}y')\in FSF^{-1}.
\]
The Mecke formula \cite[Theorem~4.4]{LastPenrose} therefore bounds the edge
intensity of the blue--blue part of $\calN$ by
\begin{equation}\label{eq:blue}
 \frac{p^2}{2}\lambda(FSF^{-1})
 =\frac{t^2}{2}\alpha_S(F).
\end{equation}

Let $\calN_{\mathrm{bad}}$ be the subgraph formed by edges with at least one
bad endpoint.  By unimodularity, the mass-transport principle bounds its edge
intensity by the intensity of its incidences at bad vertices.  Since every bad
anchor has singleton fibre,
$\deg_{\calN}(b)\leq\deg_{\calH}(b)$ for $b\in\calB$.  Conditioning on the
red process gives
\begin{align*}
 \eint(\calN_{\mathrm{bad}})
 \leq \bbE\sum_{x\in\calX\cap U}
       \one_{\{\calY\cap xF^{-1}=\varnothing\}}\deg_{\calH}(x)
 =e^{-t}\bbE\sum_{x\in\calX\cap U}\deg_{\calH}(x)
 =2he^{-t}.
\end{align*}
Since \(\calN\) is a connected factor graph of
\(\widehat{\calC}\cong\Omega\), \eqref{eq:cost} gives
\[
 c_{\mathrm P}(G)-1
 \leq \eint(\calN)-\inte(\calC)
 \leq\eint(\calN)
\leq\frac12t^2\alpha_S(F)+2he^{-t}.
\]
Thus \Cref{eq:estimate} holds with $A=2h$.

For fixed $t>0$, apply \Cref{eq:estimate} to $F_n$ and let $n\to\infty$.
Then $c_{\mathrm P}(G)-1\leq Ae^{-t}$; letting $t\to\infty$ gives
$c_{\mathrm P}(G)\leq1$.  Poisson processes have maximal cost
\cite[Theorem~1.2]{AbertMellick}, and every free pmp action is isomorphic to a
point process \cite[Theorem~1.1]{AbertMellick}.  Since $G$ is noncompact,
 every free pmp action of $G$ has cost at least one, so $G$ has fixed price
one.

Finally, let $\Gamma<G$ be a lattice, with
$\Gamma\curvearrowright(X,\mu)$ an arbitrary free pmp action.  By \cite[Example~2.14]{FraczykMellickWilkens}
\[
 \cost\bigl(\operatorname{Ind}_{\Gamma}^{G}X\bigr)-1
 =\frac{\cost(\Gamma\curvearrowright X)-1}{\covol(\Gamma)}.
\]
The induced action is free, so its cost is one, and so
$\cost(\Gamma\curvearrowright X)=1$, and
$\Gamma$ has fixed price one.
\end{proof}

\subsection{Metric form and higher-rank groups}

\begin{proof}[Proof of \Cref{cor:lc-metric}]
We have
\(
        B(R)SB(R)^{-1}\subseteq B(2R+D),
\)
so \Cref{thm:lc-criterion} applies along a sequence realizing the liminf.
\end{proof}

\begin{proof}[Proof of \Cref{cor:local-field-semisimple}]

Let \(r=\operatorname{rank}_k\mathbf G\).  Choose a maximal compact subgroup
\(K<G\), a positive Weyl chamber \(\mathfrak a^+\), and let
\(\mu:G\to\mathfrak a^+\) be the Cartan projection.  Let \(2\rho\) denote the
sum of the positive roots with multiplicities, and put
\(
        L(g)=\langle 2\rho,\mu(g)\rangle .
\)
By the Cartan integration formula in the Archimedean case, and by Macdonald's double-coset volume estimate in the
non-Archimedean case, see, for example, \cite{BenoistQuint,Macdonald}, defining \(F_R:=\{g\in G:L(g)\le R\}\), we get that
\(
\lambda(F_R)\asymp e^R R^{r-1}.
\)
  Since $r \geq 2$,
\Cref{thm:lc-criterion} applies.

\end{proof}

\begin{proof}[Proof of \Cref{cor:affine-buildings}]
Let \(C\) be the chamber set, fix \(c\in C\), and let \(\delta\) be the Weyl
distance.  If \(w=s_1\cdots s_n\) is reduced, put
\[
        q_w=q_{s_1}\cdots q_{s_n},\qquad L(w)=\log q_w .
\]
This is well defined, \(q_s>1\), and
\(
        |\{d\in C:\delta(c,d)=w\}|=q_w
\)
by \cite[(1.4) and Theorem~2.1]{AbramenkoParkinsonVanMaldeghem}.  Since the
affine Weyl group \(W\) is a finite extension of a translation lattice
\(\Lambda\simeq\mathbb Z^r\) \cite[Chapter~4]{HumphreysReflectionGroups}, the
restriction of \(L\) to \(\Lambda\) is, up to bounded error, a positive
polyhedral norm.  Hence lattice counting gives
\[
        \sum_{L(w)\le T} q_w \asymp e^T T^{r-1}.          \tag{1}
\]

The quantity \(d_q(a,b)=L(\delta(a,b))\) is an
\(\operatorname{Aut}(X)\)-invariant metric on chambers, and by (1)
\[
        |B_q(c,T)|\asymp e^T T^{r-1}.                    \tag{2}
\]
Let \(K=G_c\), a compact open subgroup, and set
\(
        F_T=\{g\in G:d_q(c,gc)\le T\}.
\)
Since \(G\) has finitely many chamber orbits, \(Gc\) is coarsely dense in \(C\);
bounded valence and (2) therefore imply
\[        \lambda(F_T)\asymp e^T T^{r-1}.                  \tag{3}
\]

For \(D=\max_{s\in S}d_q(c,sc)\), the triangle inequality gives
\(
F_TSF_T^{-1}\subseteq F_{2T+D}.
\)
Hence (3) yields
\[
\frac{\lambda(F_TSF_T^{-1})}{\lambda(F_T)^2}
\ll T^{1-r}\longrightarrow0,
\]
and \Cref{thm:lc-criterion} applies.
\end{proof}

\subsection{Direct products}

\begin{proof}[Proof of \Cref{cor:lc-products}]
Apply \Cref{lem:product-neighbourhoods} and \Cref{thm:lc-criterion} with $S = K_1 \times K_2$.  
\end{proof}

\end{document}